\documentclass{amsart}
\usepackage{graphicx}
\usepackage{amsfonts}
\newtheorem{tm}{Theorem}

\newtheorem{defi}{Definition}

\newtheorem{rem}{Remark}
\newtheorem{rems}{Remarks}
\newtheorem{lm}{Lemma}
\newtheorem{ex}{Example}
\newtheorem{observation}{Observation}

\newtheorem{prop}{Proposition}

\begin{document}

\title{Descartes' rule of signs, Rolle's theorem 
and sequences of admissible pairs}
\author{Hassen Cheriha, Yousra Gati and Vladimir Petrov Kostov}
\address{Universit\'e C\^ote d'Azur, LJAD, France and 
University of Carthage, EPT - LIM, Tunisia}
\email{hassen.cheriha@gmail.com, hassan.cheriha@unice.fr}
\address{University of Carthage, EPT - LIM, Tunisia}
\email{yousra.gati@gmail.com} 
\address{Universit\'e C\^ote d'Azur, LJAD, France} 
\email{vladimir.kostov@unice.fr}
\begin{abstract}
Given a real univariate degree $d$ polynomial $P$, the numbers $pos_k$ and 
$neg_k$ of positive and negative roots of $P^{(k)}$, $k=0$, $\ldots$, $d-1$, 
must be admissible, i.e. they must satisfy certain inequalities 
resulting from Rolle's theorem and from 
Descartes' rule of signs. For $1\leq d\leq 5$, we give the answer to the 
question for which admissible $d$-tuples of pairs $(pos_k$, $neg_k)$ 
there exist polynomials $P$ with all nonvanishing coefficients 
such that for $k=0$, $\ldots$, $d-1$, 
$P^{(k)}$ has exactly $pos_k$ positive 
and $neg_k$ negative roots all of which are simple.  


\end{abstract}
\maketitle 

\section{Introduction}

We consider real univariate polynomials and the possible numbers of real  
positive and negative roots for them and for their derivatives. Without 
loss of generality we consider only monic polynomials and we limit ourselves 
to the generic case when neither of the coefficients of the polynomial is $0$, 
i.e. we consider the family of polynomials $P:=x^d+a_{d-1}x^{d-1}+\cdots +a_0$, 
$x$, $a_j\in \mathbb{R}^*$.    

Denote by $c$ and $p$ the numbers of sign changes and sign preservations in the 
sequence $(1$, $a_{d-1}$, $\ldots$, $a_0)$ and by $pos$ and $neg$ 
the numbers of positive and negative roots of $P$ counted with multiplicity. 
Descartes' rule of signs, completed by an observation made by Fourier 
(see \cite{AJS}, \cite{Ca}, \cite{Fo} and \cite{Ga}), 
states that 

\begin{equation}\label{Descartes1}
pos\leq c~~~{\rm and}~~~c-pos\in 2\mathbb{Z}~.
\end{equation}
Applying this rule to the polynomial $P(-x)$ one gets 

\begin{equation}\label{Descartes2}
neg\leq p~~~{\rm and}~~~p-neg\in 2\mathbb{Z}~.
\end{equation}
Notice that without the assumption the coefficients $a_j$ to be nonzero 
conditions (\ref{Descartes2}) do not hold true -- for the polynomial 
$x^2-1$ one has $c=1$, $p=0$ and $neg=1$. It is clear that 

\begin{equation}\label{Descartes3}
{\rm sgn}~a_0=(-1)^{pos}~.
\end{equation}

\begin{defi}
{\rm A {\em sign pattern (SP)} of length $d+1$ is a finite sequence of plus 
and/or minus 
signs. (As we consider only monic polynomials, the first sign is a $+$.) 
We say that the sequence $(1$, $a_{d-1}$, $\ldots$, $a_0)$ defines 
the sign pattern $\sigma$ if 
$\sigma =(+$, sgn$(a_{d-1})$, $\ldots$, sgn$(a_0))$. For a given sign pattern 
$\sigma$ with $c$ sign changes and $p$ sign preservations, we call the pair 
$(c,p)$ the {\em Descartes' pair} of $\sigma$ and we say that a pair 
$(pos, neg)$ is {\em admissible} for $\sigma$ if the conditions 
(\ref{Descartes1}) and (\ref{Descartes2}) are satisfied. We say that a given 
couple (sign pattern, admissible pair) ((SP, AP)) 
is {\em realizable} if there exists 
a monic polynomial whose sequence of coefficients defines the sign pattern 
$\sigma$ and which has exactly $pos$ positive and exactly $neg$ negative roots, 
all of them distinct.}
\end{defi}      

For $d=1$, $2$ and $3$, all couples (SP, AP) are realizable 
(this is easy to check). For $d=4$, there are only two cases of couples 
(SP, AP) which are not realizable (see~\cite{Gr}):

\begin{equation}\label{Grabiner}
((+,+,-,+,+), (2,0))~~~\, \, {\rm and}~~~\, \, ((+,-,-,-,+), (0,2))~.
\end{equation}
For $d=5$, there are also only two nonrealizable couples (SP, AP), 
see~\cite{AlFu}:

\begin{equation}\label{AlbouyFu}
((+,+,-,+,-,-), (3,0))~~~\, \, {\rm and}~~~\, \, ((+,-,-,-,-,+), (0,3))~.
\end{equation}
The question which such couples are realizable is completely solved for 
$d=6$ in~\cite{AlFu}, for $d=7$ in~\cite{FoKoSh} and for $d=8$ 
partially in \cite{FoKoSh} and completely in~\cite{KoCMJ}. In \cite{KoMB} 
an example of nonrealizability is given for $d=11$ and when both components 
of the AP are nonzero.

The signs of the coefficients $a_j$ define the sign patterns $\sigma _0$, 
$\sigma _1$, $\ldots$, $\sigma _{d-1}$ corresponding to 
the polynomial $P$ and to its derivatives of order $\leq d-1$ (the SP 
$\sigma _j$ is obtained from $\sigma _{j-1}$ by deleting the last component). 
We denote by 
$(c_k, p_k)$ and $(pos_k, neg_k)$ the Descartes' and admissible pairs for 
the SPs $\sigma _k$, $k=0$, $\ldots$, $d-1$. Rolle's theorem 
implies that 

\begin{equation}\label{Rolle1}
\begin{array}{lcl}
pos_{k+1}\geq pos_k-1&,&neg_{k+1}\geq neg_k-1\\ \\ {\rm and}&
pos_{k+1}+neg_{k+1}\geq pos_k+neg_k-1~.\end{array}
\end{equation}
It can happen that $P^{(k+1)}$ has more real roots than $P^{(k)}$. 
E. g. this is the case of $P=x^3+3x^2-8x+10=(x+5)((x-1)^2+1)$, because 
$P'=3x^2+6x-8$ has one positive and one negative root. It is always true 
that 

\begin{equation}\label{Rolle2}
pos_{k+1}+neg_{k+1}+3-pos_k-neg_k\in 2\mathbb{N}~.
\end{equation}   

\begin{defi}
{\rm For a given sign pattern $\sigma _0$ of length $d+1$, 
and for $k=0$, $\ldots$, $d-1$, 
suppose that the pair 
$(pos_k,neg_k)$ satisfies the conditions (\ref{Descartes1}) -- 
(\ref{Descartes3}) and (\ref{Rolle1}) -- (\ref{Rolle2}). 
Then we say that $((pos_0,neg_0)$, $\ldots$, $(pos_{d-1},neg_{d-1}))~(*)$ is a 
{\em sequence of admissible pairs (SAP)} 
(i.e. a sequence of pairs admissible for the sign pattern $\sigma _0$ 
in the sense of these conditions). We say 
that a SAP is {\em realizable} if there exists a polynomial $P$ 
the signs of whose coefficients define the SP $\sigma _0$ and such that 
for $k=0$, $\ldots$, $d-1$, the polynomial $P^{(k)}$ has exactly $pos_k$ 
positive and $neg_k$ negative roots, all of them being simple.}
\end{defi}

\begin{rem}\label{SAPdefinesSP}
{\rm The SAP $(*)$ defines the SP $\sigma _0$. This follows from condition 
(\ref{Descartes3}). Given a SAP $((pos_0$, $neg_0)$, $\ldots$, $(pos_{d-1}$, 
$neg_{d-1}))$, the corresponding SP (beginning with a $+$) equals} 

$$(~+~,~(-1)^{pos_{d-1}}~,~(-1)^{pos_{d-2}}~,~\ldots ~,~(-1)^{pos_0}~)~.$$
{\rm However, for a given SP there are, in general, 
several possible SAPs. The following example gives an idea how fast the 
number of SAPs compatible with a given SP might grow with $d$: for $d=2$ 
and for the SP $(+,+,+)$, there are 
two possible SAPs, namely, $((0,2),(0,1))$ and $((0,0),(0,1))$. For $d=3$ and 
for the SP $(+,+,+,+)$, there are three possible SAPs: 

$$((0,3),(0,2),(0,1))~,~~~ 
((0,1),(0,2),(0,1))~~~{\rm and}~~~((0,1),(0,0),(0,1))~.$$ 
For $d=4$ and for the SP 
$(+,+,+,+,+)$, this number is $7$: 

$$\begin{array}{cccc}
((0,4),(0,3),(0,2),(0,1))&,&((0,2),(0,3),(0,2),(0,1))&,\\ 
((0,2),(0,1),(0,2),(0,1))&,&((0,2),(0,1),(0,0),(0,1))&,\\ 
((0,0),(0,3),(0,2),(0,1))&,&((0,0),(0,1),(0,2),(0,1))&\\ 
{\rm and}&& ((0,0),(0,1),(0,0),(0,1))&.\end{array}$$
The next six numbers (denoted by $A(d)$), 
obtained as numbers of SAPs compatible with 
the all-pluses SP of length $d+1$, are:

$$12~,~30~,~55~,~143~,~273~,~728~.$$
They coincide with the terms of sequence A047749 
of The On-line Encyclopedia of Integer Sequences founded by N.~J.~A.~Sloane 
in 1964. To be more precise, sequence A047749 begins like this: 
$1$, $1$, $1$, $2$, $3$, $7$, $12$, $30$, $55$, $143$, $\ldots$. Its terms are 
defined as ${3m\choose m}/(2m+1)$ if $n=2m$ and as 
${3m+1\choose m+1}/(2m+1)$ if $n=2m+1$. It would be interesting to (dis)prove 
that this formula applies to all numbers $A(d)$ for $d\in \mathbb{N}$. 
We prove a weaker statement 
(see Proposition~\ref{A(d)}) which implies that the numbers $A(d)$ 
grow faster than the numbers $[d/2]+1$ of APs $(pos_0, neg_0)$ compatible with 
the all-pluses SP of length $d+1$. These APs are $(0,d-2r)$, 
$r=0$, $\ldots$, $[d/2]$ (the integer part of $d/2$).} 
\end{rem}

\begin{prop}\label{A(d)}
For $d\geq 2$ even, one has $A(d)\geq 2A(d-1)$. For $d\geq 3$ odd, one has 
$A(d)\geq 3A(d-1)/2$.
\end{prop}

In what follows, 
for the sake of making things more explicit, we write down often 
the couples (SP, SAP), not just the SAPs.

\begin{ex}
{\rm Consider the couple (SP, AP) $C:=((+,+,-,+,+)$, $(0,2))$. It can be 
extended in two ways into a couple (SP, SAP):}

$$\begin{array}{cccccccccccc}
(&(+,+,-,+,+)&,&(0,2)&,&(2,1)&,&(1,1)&,&(0,1)&)&{\rm and}\\ \\ 
(&(+,+,-,+,+)&,&(0,2)&,&(0,1)&,&(1,1)&,&(0,1)&)&.\end{array}$$
{\rm Indeed, by Rolle's theorem, the derivative of a polynomial realizing 
the couple $C$ has at least one negative root. Condition (\ref{Descartes3}) 
implies that this derivative (which is of degree 3) has an even number of 
positive roots. This gives the two possibilities $(2,1)$ and $(0,1)$ 
for $(pos _1, neg_1)$. The second derivative has a positive and a negative 
root. Indeed, it is a degree 2 polynomial with positive leading and negative 
last coefficient. The realizability of the above two couples (SP, SAP) 
is justified in the proof of Theorem~\ref{maintm}.}
\end{ex}

Our first result is the following proposition:

\begin{prop}\label{mainprop1}
For any given SP of length $d+1$, $d\geq 1$, there exists a unique SAP such 
that $pos_0+neg_0=d$. This SAP is realizable. For the given SP, 
this pair $(pos_0, neg_0)$ is its Descartes' pair. 
\end{prop}

\begin{rems}\label{remshyp}
{\rm (1) Consider a SP of length $d+1$, $d\geq 1$, and a SAP with 
$(pos_0, neg_0)=(d-1,1)$ (resp. $(pos_0, neg_0)=(1,d-1)$). 
By Proposition~\ref{mainprop1}, this 
couple (SP, SAP) is realizable by some polynomial $P$. 
But then all other SAPs with the same 
pairs $(pos_k, neg_k)$, $k=1$, $\ldots$, $d-1$, and with 
$(pos_0, neg_0)=(d-1-2\nu ,1)$ (resp. $(pos_0, neg_0)=(1,d-1-2\nu )$), $\nu =1$, 
$\ldots$, $[(d-1)/2]$, 
are also realizable with 
this SP. Indeed, by adding a small linear term $\varepsilon x$ to the 
polynomial $P$ (without changing the SP of its coefficients) 
one can obtain the condition the critical values of $P$  
to be distinct. In the case $(pos_0, neg_0)=(1,d-1)$, the constant term of 
$P$ is negative, see (\ref{Descartes3}). Hence in the family $P-v$, $v>0$ 
(defining the same SP for all values of $v$)  
one encounters polynomials with exactly one positive and exactly $d-1$, $d-3$, 
$\ldots$, $d-2[(d-1)/2]$ negative roots for suitable values of $v$. 
In the case $(pos_0, neg_0)=(d-1,1)$, the sign of the constant term equals 
$(-1)^{d-1}$ and in the family $P+(-1)^{d-1}v$ one encounters polynomials 
with exactly one negative and exactly $d-1$, $d-3$, 
$\ldots$, $d-2[(d-1)/2]$ positive roots.

(2) In the same way, if $(pos_0, neg_0)=(d,0)$ (resp. $(pos_0, neg_0)=(0,d)$), 
then this 
couple (SP, SAP) is realizable by some polynomial $P$, and all couples 
(SP, SAP) with the same SP, the same pairs $(pos_k, neg_k)$, $k=1$, $\ldots$, 
$d-1$, and with $(pos_0, neg_0)=(d-2\nu ,0)$ (resp. 
$(pos_0, neg_0)=(0,d-2\nu )$), $\nu =1$, 
$\ldots$, $[d/2]$, are also realizable.}
\end{rems}   

There are examples of couples (SP, SAP) which are not realizable:

\begin{ex}\label{exnotrealiz}
{\rm For $d=4$, the couple (SP, SAP)}  

\begin{equation}\label{exd4}
(~(+,+,-,+,+)~,~(2,0)~,~(2,1)~,~(1,1)~,~(0,1)~)\end{equation} 
{\rm is not realizable because the first of the two couples (SP, AP) 
(\ref{Grabiner}) is not realizable. Hence 
for $d=5$, the following couples (SP, SAP) are not realizable:}

\begin{equation}\label{exd45}
\begin{array}{ccccccccccccc}
(&(+,+,-,+,+,+)&,&(2,1)&,&(2,0)&,&(2,1)&,&(1,1)&,&(0,1)&)~,\\ \\ 
(&(+,+,-,+,+,+)&,&(0,1)&,&(2,0)&,&(2,1)&,&(1,1)&,&(0,1)&)~,\\ \\ 
(&(+,+,-,+,+,-)&,&(3,0)&,&(2,0)&,&(2,1)&,&(1,1)&,&(0,1)&)~,\\ \\ 
(&(+,+,-,+,+,-)&,&(1,0)&,&(2,0)&,&(2,1)&,&(1,1)&,&(0,1)&)~.\end{array}
\end{equation}
{\rm For $d=5$, the following couple (SP, SAP) is also not realizable, 
see the first of the nonrealizable couples (SP, AP) in (\ref{AlbouyFu}):}

\begin{equation}\label{exd5}
(~(+,+,-,+,-,-)~,~(3,0)~,~(3,1)~,~(2,1)~,~(1,1)~,~(0,1)~)~.
\end{equation}
\end{ex}

In what follows we reduce by half the cases to be considered 
using the following fact: 

\begin{observation}\label{obs1}
{\rm If $a$ is a root of the polynomial $P(x)$, then $-a$ is a root of $P(-x)$. 
Hence if $P(x)$ has $pos$ positive and $neg$ negative roots, then $P(-x)$ has 
$neg$ positive and $pos$ negative roots.}
\end{observation}

\begin{rems}\label{remsobs1}
{\rm (1) Observation \ref{obs1} allows to consider for every couple 
of polynomials $(P(x)$, 
$(-1)^dP(-x))$ 
only one of them. We choose this to be the one with sgn$(a_{d-1})=+$. We say 
that the polynomials $P(x)$ and $P(-x)$ are {\em equivalent modulo the 
$\mathbb{Z}_2$-action}. 

(2) When couples (SP, AP) are studied, one can use a 
second symmetry to reduce the number of cases to be considered. This symmetry 
stems from the fact that the polynomials $P(x)$ and its reverted one 
$($sgn$(a_0))x^dP(1/x)$ have one and 
the same numbers of positive and negative roots. Up to a sign, 
the SP defined by the latter 
polynomial is the one defined by $P$, but read backward. 
In the present paper we 
cannot use reversion, because the two ends of a SP do not play the same role -- 
we differentiate w.r.t. of $x$ which makes disappear 
one by one the coefficients of the lowest 
degree monomials.}
\end{rems}  

In the present paper we prove the following theorem:

\begin{tm}\label{maintm}
(1) For $d=1$, $2$ and $3$, all couples (SP, SAP) are realizable.

(2) For $d=4$, the only couple (SP, SAP) which is not 
realizable is~(\ref{exd4}). 

(3) For $d=5$, the only couples (SP, SAP) which are not realizable are 
(\ref{exd45}) and~(\ref{exd5}).
\end{tm}

\begin{rem}
{\rm As we see, for degrees up to $5$, the questions of realizability of 
couples (SP, AP) and (SP, SAP) (or just SAP, see Remark~\ref{SAPdefinesSP}) 
have the same 
answers. The much more numerous cases of SAPs compared to couples 
(SP, AP) as $d$ grows (see Remark~\ref{SAPdefinesSP} and 
Proposition~\ref{A(d)}) 
indicate that it is not unlikely 
these answers to be different for some $d\geq 6$.}
\end{rem} 

In the proof of Theorem~\ref{maintm} we use the following proposition: 

\begin{prop}\label{mainprop2}
Suppose that the couple ($\sigma$, $U$) is realizable by a polynomial $P$, 
where 
$\sigma$ is a sign pattern of length $d+1$ and $U$ is a SAP. 
Denote by $\sigma ^*$ 
(resp. by $\sigma ^{\dagger}$) the SP of length $d+2$ obtained from 
$\sigma$ by adding a sign $+$ (resp. $-$) to its right. Then 

(1) for $d$ even, the couple ($\sigma ^*$, $((0,1),~U)$) 
(resp. ($\sigma ^{\dagger}$, $((1,0),~U)$)) is realizable.

(2) for $d$ odd, the couple ($\sigma ^*$, $((0,0),~U)$) 
(resp. ($\sigma ^{\dagger}$, $((1,1),~U)$)) is realizable.
\end{prop}

Another proposition which implies part of the proof of Theorem~\ref{maintm} 
reads:

\begin{prop}\label{mainprop3}
For $d=5$, consider the SAPs in which $(pos_2, neg_2)=(0,1)$ or $(1,0)$. 
All these SAPs are realizable (with the SPs which they define, see 
Remark~\ref{SAPdefinesSP}).
\end{prop}

The following lemma allows to construct examples of realizability of couples 
(SP, SAP) by deforming polynomials with multiple roots.

\begin{lm}\label{lm23}
Consider the polynomials $S:=(x+1)^3(x-a)^2$ and $T:=(x+a)^2(x-1)^3$, $a>0$. 
Their coefficients of $x^4$ are  
positive if and only if respectively $a<3/2$ and $a>3/2$. 
The coefficients of the 
polynomial $S$ define the SP

$$\begin{array}{lllc}
(+,+,+,+,-,+)&{\rm for}&a\in (~0~,~(3-\sqrt{6})/3~)&,\\ \\  
(+,+,+,-,-,+)&{\rm for}&a\in (~(3-\sqrt{6})/3~,~3-\sqrt{6}~)&,\\ \\ 
(+,+,-,-,-,+)&{\rm for}&a\in (~3-\sqrt{6}~,~2/3~)&{\rm and}\\ \\ 
(+,+,-,-,+,+)&{\rm for}&a\in (~2/3~,~3/2~)&.\end{array}$$
The coefficients of $T$ define the SP

$$\begin{array}{lllc}
(+,+,-,+,+,-)&{\rm for}&a\in (~3/2,~(3+\sqrt{6})/3~)&,\\ \\  
(+,+,-,-,+,-)&{\rm for}&a\in (~(3+\sqrt{6})/3~,~3+\sqrt{6}~)&{\rm and}\\ \\ 
(+,+,+,-,+,-)&{\rm for}&a>3+\sqrt{6}&.\end{array}$$
\end{lm}

Finally, we make use of two more propositions to prove Theorem~\ref{maintm}:

\begin{prop}\label{mainprop4}
For $d=5$, all SAPs with $pos_1+neg_1=4$ and with the exception 
of the one defined by (\ref{exd5}) are realizable.
\end{prop}

\begin{prop}\label{mainprop5}
For $d=5$, all SAPs with $pos_1+neg_1=2$ and with the exception 
of the four SAPs defined by (\ref{exd45}) are realizable.
\end{prop}

We present all proofs in Section~\ref{secproofs} in the following order: 
we begin with the proof of part (1) of Theorem~\ref{maintm}. Then 
we prove Propositions~\ref{mainprop1} and \ref{mainprop2}, 
then we give the proof of part (2) of Theorem~\ref{maintm}, after this the 
proofs of Proposition~\ref{mainprop3}, Lemma~\ref{lm23}, 
Proposition~\ref{mainprop4} and Proposition~\ref{mainprop5}, and we finish 
with the proofs of part (3) of Theorem~\ref{maintm} and of 
Proposition~\ref{A(d)}. In the proofs of 
Propositions~\ref{mainprop4} and \ref{mainprop5}, when a given case is 
realizable by a given polynomial, we list in a line the approximations 
of the real roots of the polynomial and its first three derivatives. The roots 
of one and the same derivative are separated by commas, between the roots 
of the different derivatives we put semicolons. We do not give the 
roots of the fourth derivatives which are always negative, see 
Observation~\ref{obs1} and part (1) of Remarks~\ref{remsobs1}. 
\vspace{1mm}

{\bf Acknowledgement.} The subject of the present paper is a direct 
continuation of the common work of the third author with Boris Shapiro and Jens 
Forsg{\aa}rd from the University of Stockholm on sign patterns and admissible 
pairs. The third author expresses his gratitude to the Universities of 
Stockholm and Carthage for their kind hospitality, and also to 
Groupement Euro-Maghr\'ebin de 
Math\'ematiques et leurs Int\'eractions of CNRS for partially supporting his 
stay at the University of Carthage.

\section{Proofs\protect\label{secproofs}}

\begin{proof}[Proof of part (1) of Theorem~\ref{maintm}]

For $d=1$, the only possible couple (SP, SAP) modulo the 
$\mathbb{Z}_2$-action and an example of a 
polynomial which realizes it is:

$$
((+,+), (0,1))~~~\, \, {\rm realizable~by}~~~\, \, x+1~~~.
$$
For $d=2$, there are three such couples (we list also the derivatives):

$$\begin{array}{cccccc}{\rm (SP,~SAP)}&&P&&P'\\ \\ 
((+,+,+), (0,2), (0,1))&&(x+1)(x+2)=x^2+3x+2&&2x+3&,\\ \\ 
((+,+,+), (0,0), (0,1))&&(x+1)^2+1=x^2+2x+2&&2x+2&{\rm and}\\ \\ 
((+,+,-), (1,1), (0,1)&&(x+2)(x-1)=x^2+x-2&&2x+1&.\end{array}$$
For $d=3$, there are $10$ such couples (we list them together with 
$P$, $P'$ and $P''$): 

$$\begin{array}{rcll}
((+,+,+,+)&,& (0,3), (0,2), (0,1))&\\ \\ 
x^3+6x^2+11x+6=&&3x^2+12x+11=&6x+12=\\ 
(x+1)(x+2)(x+3)&&3(x+2+1/\sqrt{3})(x+2-1/\sqrt{3})&6(x+2)\end{array}$$ 
$$\begin{array}{rcll}
((+,+,+,+)&,& (0,1), (0,2), (0,1))&\\ \\ 
x^3+5x^2+8x+6=&&3x^2+10x+8=&6x+10=\\ 
(x+3)((x+1)^2+1)&&3(x+2)(x+4/3)&6(x+5/3)\end{array}$$
$$\begin{array}{rcll}
((+,+,+,+)&,& (0,1), (0,0), (0,1))&\\ \\ 
x^3+3x^2+13x+11=&&3x^2+6x+13=&6x+6=\\  
(x+1)((x+1)^2+10)&&3(x+1)^2+10&6(x+1)\end{array}$$ 
$$\begin{array}{rcll}
((+,+,+,-)&,& (1,2), (0,2), (0,1))&\\ \\ 
x^3+4x^2+x-6=&&3x^2+8x+1=&6x+8=\\ 
(x+3)(x+2)(x-1)&&3(x+\frac{4+\sqrt{13}}{3})(x+\frac{4-\sqrt{13}}{3})&
6(x+4/3)\end{array}$$
$$\begin{array}{rcll}
((+,+,+,-)&,& (1,0), (0,2), (0,1))&\\ \\ 
x^3+3x^2+x-5=&&3x^2+6x+1=&6x+6=\\ 
(x-1)((x+2)^2+1)&&3(x+1+\sqrt{2/3})(x+1-\sqrt{2/3})&6(x+1)\end{array}$$ 
$$\begin{array}{rcll}
((+,+,+,-)&,& (1,0), (0,0), (0,1))&\\ \\ 
x^3+3x^2+4x-8=&&3x^2+6x+4=&6x+6=\\ 
(x-1)((x+2)^2+4)&&3(x+1)^2+1&6(x+1)\end{array}$$ 
$$\begin{array}{rcll}
((+,+,-,+)&,& (2,1), (1,1), (0,1))&\\ \\ 
x^3+x^2-10x+8=&&3x^2+2x-10=&6x+2=\\ 
(x-1)(x-2)(x+4)&&3(x+\frac{1-\sqrt{31}}{3})(x+\frac{1+\sqrt{31}}{3})&
6(x+1/3)\end{array}$$ 
$$\begin{array}{rcll}
((+,+,-,+)&,& (0,1), (1,1), (0,1))&\\ \\ 
x^3+2x^2-6x+8=&&3x^2+4x-6=&6x+4=\\ 
((x-1)^2+1)(x+4)&&3(x+\frac{2-\sqrt{22}}{3})
(x+\frac{2+\sqrt{22}}{3})&6(x+2/3)\end{array}$$
$$\begin{array}{rcll}
((+,+,-,-)&,& (1,2), (1,1), (0,1))&\\ \\ 
x^3+x^2-4x-4=&&3x^2+2x-4=&6x+2=\\ 
(x-2)(x+1)(x+2)&&3(x+\frac{1-\sqrt{13}}{3})(x+\frac{1+\sqrt{13}}{3})&6(x+1/3)
\end{array}$$

$$\begin{array}{rcll}
((+,+,-,-)&,& (1,0), (1,1), (0,1))&\\ \\ 
x^3+x^2-0.5x-1.5=&&3x^2+2x-0.5=&6x+2=\\ 
(x-1)((x+1)^2+0.5)&&
3(x+\frac{1+\sqrt{2.5}}{3})(x+\frac{1-\sqrt{2.5}}{3})&6(x+1/3)
\end{array}$$
\end{proof}

\begin{proof}[Proof of Proposition~\ref{mainprop1}]
The condition $pos_0+neg_0=d$ implies that if a polynomial $P$ realizes 
a SAP with the given SP, then $pos_0=c$ and $neg_0=p$, i.e. the admissible pair 
$(pos_0,neg_0)$ is the Descartes' pair for the given SP. Next, one has 
$pos_1\geq pos_0-1$ and $neg_1\geq neg_0-1$, see (\ref{Rolle1}). As deg$P'=d-1$, 
this means that $pos_1+neg_1\geq d-2$, i.e. at least $d-2$ of the roots of the 
polynomial $P'$ are real. So the remaining one root is also real (hence 
$pos_1+neg_1=d-1$) and 
its sign is defined by condition (\ref{Descartes3}). 
Continuing like this one proves uniqueness of the SAP satisfying the condition 
$pos_0+neg_0=d$. 

Now we show by induction on $d$ that any given SP is realizable with its 
Descartes' pair. For $d=1$ this is evident. 
Suppose that a sign pattern $\sigma$ of length $d+1$ is realizable with its 
Descartes' pair by a polynomial $P$. Denote by $\kappa$ the last component of 
$\sigma$ (hence $\kappa =+$ or $\kappa =-$). 
Consider the sign patterns $\sigma ^*$ 
and $\sigma ^{\dagger}$ defined in Proposition~\ref{mainprop2}. For 
$\varepsilon >0$ small enough, the polynomial $P(x)(x+\varepsilon )$ 
defines the sign pattern $\sigma ^*$ for $\kappa =+$ and $\sigma ^{\dagger}$ for 
$\kappa =-$, and vice versa for $P(x)(x-\varepsilon )$. Indeed, for 
$\varepsilon$ small enough, the 
coefficients of $x^{d+1}$, $x^d$, $\ldots$, $x$ of $P(x)(x\pm \varepsilon )$ 
have the same signs as the coefficients of 
$x^{d}$, $x^{d-1}$, $\ldots$, $1$ of $P$ (because the former equal $1$, 
$a_{d-1}\pm \varepsilon$, $a_{d-2}\pm \varepsilon a_{d-1}$, 
$\ldots$, $a_0\pm \varepsilon a_1$). The sign of the last coefficient 
equals $\pm \kappa$ in the case of $P(x)(x\pm \varepsilon )$. Thus one realizes 
the SPs $\sigma ^*$ and $\sigma ^{\dagger}$ of length $d+2$. 

\end{proof}

\begin{proof}[Proof of Proposition~\ref{mainprop2}]
Denote by $Q$ some polynomial such that $Q'=P$. Suppose that $d$ is even. 
Then for $A>0$ sufficiently large, the polynomial $Q+A$ (resp. 
$Q-A$) has a single real root which is simple and negative (resp. simple and 
positive), so $Q+A$ realizes the SAP 
$((0,1),~U)$ with the SP $\sigma ^*$ (resp. $Q-A$ realizes the SAP $((1,0),~U)$ 
with the SP $\sigma ^{\dagger}$). 

Suppose that $d$ is odd. 
Then for $A>0$ sufficiently large, the polynomial $Q+A$ has no real roots 
and realizes the SAP $((0,0),~U)$ with the SP $\sigma ^*$ 
(resp. the polynomial $Q-A$ has a single positive and a single negative root, 
both simple, so it realizes the SAP $((1,1),~U)$ 
with the SP $\sigma ^{\dagger}$.  
\end{proof}

\begin{proof}[Proof of part (2) of Theorem~\ref{maintm}]

We make use of Propositions~\ref{mainprop2} and \ref{mainprop1} 
and of Remarks~\ref{remshyp}. 
Hence when the admissible pair for $P$ is of the form $(1,1)$  
or $(0,0)$, then realizability of the SAP follows from 
Proposition~\ref{mainprop2}. When $pos_0+neg_0=4$, realizability follows from 
Proposition~\ref{mainprop1}. When the Descartes pair of the SP equals $(0,4)$ 
and $(pos_0, neg_0)=(0,2)$, realizability follows from  Remarks~\ref{remshyp}. 
We present the proof of realizability of the 
remaining cases by listing the SPs in the lexicographic order. In the proof 
$\varepsilon$ and $\eta$ denote positive and sufficiently small numbers. 
\vspace{1mm}


{\bf 1.} $((+,+,+,+,+), (0,2), (0,1), (0,2), (0,1))$. We set 
$P'':=(x+1)^2-\varepsilon$. 
Hence $P''$ has two negative roots and 
$P'''$ has a simple negative root. Set $P':=\int _{-2}^xP''(t)dt$. 
Hence $P'(0)>0$ 
and $P'$ has a single root which equals $-2$. Then we set 
$P:=\int _{-2-\eta}^xP'(t)dt$.
\vspace{1mm}

{\bf 2.} $((+,+,+,+,+), (0,2), (0,1), (0,0), (0,1))$. For $x\in [-3,-0.5]$, 
the graphs of the polynomial 
$P^{\ddagger}:=(x+1)(x+2)(1+\varepsilon x^2)$ and of its first and second 
derivatives are close to the graphs respectively of $(x+1)(x+2)=x^2+3x+2$, 
$2x+3$ and $2$. It is clear that $P^{\ddagger}$ has a complex conjugate pair 
of roots. As 

$$(P^{\ddagger})'=(2x+3)(1+\varepsilon x^2)+2\varepsilon x(x+1)(x+2)=
2x+3+2\varepsilon x(2x+1)(x+1)~,$$
for $\varepsilon >0$ small enough, the polynomial $(P^{\ddagger})'$ has a single 
real root which is close to $-3/2$, and 
$(P^{\ddagger})''=2(1+\varepsilon (6x^2+6x+1))$ has no real root. Obviously, 
$(P^{\ddagger})'''=\varepsilon (12x+6)$ has one negative root.  
\vspace{1mm}

{\bf 3.} $((+,+,+,-,+), (2,0), (1,2), (0,2), (0,1))$. One sets 

$$P':=(x-0.25)((x+1)^2-\varepsilon )=x^3+1.75x^2+0.5x-0.25+O(\varepsilon )~,$$
and then $P=\int _{0.25}^xP'(t)dt-\eta$. 
\vspace{1mm}
 
{\bf 4.} $((+,+,+,-,+), (2,0), (1,0), (0,2), (0,1))$. We set 
$P'':=(x+1)^2-\varepsilon$, 
$P':=\int _1^xP''(t)dt$ and $P:=\int _1^xP'(t)dt-\eta$. 
\vspace{1mm}

{\bf 5.} $((+,+,+,-,+), (2,0), (1,0), (0,0), (0,1))$. We set 
$P:=x^4-x+\varepsilon +\eta x^2+\eta ^2x^3$. Hence $P''=12x^2+6\eta ^2x+2\eta$  
has no real root and $P'''=24x+6\eta ^2$ has a negative root.
The polynomial $T:=x^4-x+\varepsilon$ has 
two positive roots and a complex conjugate pair, so for 
$0<\eta \ll \varepsilon$ this is also the case of $P$. As for $T'$, 
it has a single real root $1/4^{1/3}$, so $P'$ has a single real root close 
to $1/4^{1/3}$.   
\vspace{1mm}

{\bf 6.} $((+,+,+,-,+), (0,2), (1,2), (0,2), (0,1))$. Set 

$$P':=(x-0.5)(x+1)(x+3)=x^3+3.5x^2+x-1.5~.$$
One has $|\int _{-3}^{-1}P'(t)dt|>|\int _{-1}^{0.5}P'(t)dt|$, 
because the graph of $P'$ is symmetric w.r.t. the point $(-7/6,P'(-7/6))$ 
with $P'(-7/6)>0$. Hence $P$ has minima at $-3$ and $0.5$ and $P(-3)<P(0.5)$. 
Thus one can choose $a\in \mathbb{R}$ such that $P:=\int _0^xP'(t)dt+a$ 
two negative simple roots and no nonnegative root.
\vspace{1mm}

{\bf 7.} $((+,+,-,+,+), (0,2), (2,1), (1,1), (0,1))$. One sets 

$$P':=(x+3)((x-1)^2-\varepsilon )=x^3+x^2-5x+3+O(\varepsilon )~~~{\rm and}~~~
P:=\int _{-3-\eta}^xP'(t)dt~.$$

{\bf 8.} $((+,+,-,+,+), (0,2), (0,1), (1,1), (0,1))$. One sets 

$$P':=(x+1)((x-0.25)^2+\varepsilon )=x^3+0.5x^2-0.25x+0.0625+O(\varepsilon )~~~
{\rm and}~~~P:=\int _{-1}^xP'(t)dt-\eta ~.$$

{\bf 9.} $((+,+,-,-,+), (2,0), (1,2), (1,1), (0,1))$. One sets 

$$P':=(x-1.5)((x+1)^2-\varepsilon )=x^3+0.5x^2-2x-1.5+O(\varepsilon )~~~
{\rm and}~~~P:=\int _{1.5}^xP'(t)dt-\eta ~.$$ 

{\bf 10.} $((+,+,-,-,+), (2,0), (1,0), (1,1), (0,1))$. One sets

$$P':=(x-1)((x+1)^2+\varepsilon )=x^3+x^2-x-1+O(\varepsilon )~~~{\rm and}~~~
P:=\int _{1}^xP'(t)dt-\eta ~.$$

{\bf 11.} $((+,+,-,-,+), (0,2), (1,2), (1,1), (0,1))$. One sets

$$P':=(x-1)(x+2)(x+\varepsilon )=x^3+(1+\varepsilon )x^2+
(-2+\varepsilon )x-2\varepsilon ~.$$
Thus $|\int _{-2}^{-\varepsilon }P'(t)dt|>|\int _{-\varepsilon }^1P'(t)dt|$. 
One can choose $\eta$ such that for $P:=\int _{-2-\eta }^xP'(t)dt$ one has 
$P(0)>0$ and $P$ has two negative and no nonnegative root.

\end{proof}

\begin{proof}[Proof of Proposition~\ref{mainprop3}]
First of all we explicit the SAPs with $(pos_2, neg_2)=(1,0)$ or $(1,0)$. 
It is clear that when the SP $\sigma _0$ begins with two signs $+$, then 
for $((pos_3, neg_3)$, $(pos_4,neg_4))$ 
one has the three possibilities 

\begin{equation}\label{threepos}
((0,2),(0,1))~~~\, ,~~~\, ((1,1),(0,1))~~~\, {\rm and}~~~\, 
((0,0),(0,1))~.
\end{equation}
Proposition~\ref{mainprop2} allows not to consider the case 
$(pos_0, neg_0)=(0,1)$ or $(1,0)$ because then  
the couple (SP, SAP) is realizable. In particular, one needs not to consider 
the situation when $(pos_1, neg_1)=(0,0)$, because then $(pos_0, neg_0)=(0,1)$ 
or $(1,0)$, see (\ref{Rolle1}) and (\ref{Rolle2}). 
Therefore if $(pos_2, neg_2)=(1,0)$, 
then there exist the following four possible choices for 
$((pos_0, neg_0)$, $(pos_1, neg_1))$: 

\begin{equation}\label{fourpos1}
((3,0),(2,0))~~~\, ,~~~\, ((2,1),(2,0))~~~\, ,~~~\, 
((2,1),(1,1))~~~\, {\rm and}~~~\, ((1,2),(1,1))~.
\end{equation}
For $(pos_2, neg_2)=(0,1)$, the possibilities are also four:

\begin{equation}\label{fourpos2}
((0,3),(0,2))~~~\, ,~~~\, ((1,2),(0,2))~~~\, ,~~~\, 
((1,2),(1,1))~~~\, {\rm and}~~~\, ((2,1),(1,1))~.
\end{equation}
Combining the possibilities (\ref{threepos}) with each of the choices 
(\ref{fourpos1}) (resp. (\ref{fourpos2})) one obtains $12$ SAPs with 
$(pos_2, neg_2)=(1,0)$ and $12$ with $(pos_2, neg_2)=(0,1)$.

To realize a SAP with $(pos_2, neg_2)=(1,0)$  
we consider the polynomial 
$T:=x^3-1$ having a single real root $1$. 
If we choose $P''$ to equal $T$, 
and $P'$ to equal $x^4/4-x+0.1$, then $P'$ has two positive roots 
$\lambda _1:=0.10\ldots$ and $\lambda _2:=1.55\ldots$ 
and a complex conjugate pair. One can represent $P$ in the form 
$\int _{\lambda _1}^xP'(t)dt+\varepsilon$. For $\varepsilon =0$, it has a double 
root at $\lambda _1$, a simple one $>\lambda _1$ and a complex conjugate pair. 
Hence for $\varepsilon >0$ small enough, it has three positive simple roots 
and a conjugate pair. 

Finally we set $P:=\int _{\lambda _1}^xP'(t)dt+\varepsilon +
\theta _1x^4+\theta _2x^3$, where $\theta _j\in \mathbb{R}^*$ are small enough 
(much smaller than $\varepsilon$) 
and such that the polynomial $P'''$ realizes the necessary couple 
(\ref{threepos}). The sign pattern begins with two signs $+$, so one should 
have $\theta _1>0$. It is clear that $P$ realizes the SAP whose first three 
APs are $(3,0)$, $(2,0)$ and $(1,0)$. 

If one sets $P:=\int _{\lambda _2}^xP'(t)dt-\varepsilon +
\theta _1x^4+\theta _2x^3$, then the real roots of  
$P|_{\varepsilon =\theta _1=\theta _2=0}$ are $-0.96\ldots$ (simple) and $\lambda _2$ 
(double), so $P$ realizes the SAP whose first three 
APs are $(2,1)$, $(2,0)$ and $(1,0)$. 

If one sets $P'':=T$ and $P':=x^4/4-x-0.1$, then the real roots of $P'$ are 
$\mu _1:=-0.099\ldots$ and $\mu _2:=1.6\ldots$. If we set 
$P:=\int _{\mu _2}^xP'(t)dt-\varepsilon +
\theta _1x^4+\theta _2x^3$, then $P$ realizes the SAP whose first three 
APs are $(2,1)$, $(1,1)$ and $(1,0)$. If we set 
$P:=\int _{\mu _1}^xP'(t)dt+\varepsilon +
\theta _1x^4+\theta _2x^3$, then $P$ realizes the SAP whose first three 
APs are $(1,2)$, $(1,1)$ and $(1,0)$.

To realize a SAP with $(pos_2, neg_2)=(0,1)$  
we consider the polynomial 
$U:=x^3+1$ having a single real root $(-1)$. 
By analogy we set $P'':=U$ and obtain the polynomial $P'=x^4/4+x-0.1$ having 
roots $\nu _1:=-1.6\ldots =-\mu _2$ and $\nu _2:=0.09\ldots =-\mu _1$. Then 
$P:=\int _{\nu _1}^xP'(t)dt+\varepsilon +
\theta _1x^4+\theta _2x^3$ realizes the SAP whose first three 
APs are $(1,2)$, $(1,1)$ and $(0,1)$, and 
$P:=\int _{\nu _2}^xP'(t)dt-\varepsilon +
\theta _1x^4+\theta _2x^3$ realizes the SAP whose first three 
APs are $(2,1)$, $(1,1)$ and $(0,1)$. 

If we set $P'':=U$ and $P'=x^4/4+x+0.1$, then the roots of $P'$ equal 
$\rho _1:=-1.5\ldots =-\lambda _2$ and $\rho _2:=-0.1\ldots =-\lambda _1$. Thus 
$P:=\int _{\rho _1}^xP'(t)dt+\varepsilon +
\theta _1x^4+\theta _2x^3$ realizes the SAP whose first three 
APs are $(2,1)$, $(1,1)$ and $(0,1)$, and 
$P:=\int _{\rho _2}^xP'(t)dt-\varepsilon +
\theta _1x^4+\theta _2x^3$ realizes the SAP whose first three 
APs are $(0,3)$, $(0,2)$ and $(0,1)$.
\end{proof}

\begin{proof}[Proof of Lemma~\ref{lm23}]
The proof of the lemma is straightforward -- we list the coefficients 
of the polynomials $S$ and $T$ (without the leading one) and below them 
their roots. For the polynomial $S$, the list looks like this:

$$\begin{array}{ccccccccc}3-2a&,&3-6a+a^2&,&1-6a+3a^2&,&-2a+3a^2&,&a^2\\ \\ 
3/2&&3\pm \sqrt{6}&&(3\pm \sqrt{6})/3&&0~,~2/3&&0\end{array}$$
and one has the following order of these roots on the real line 
(we list the roots and their approximative values):

$$\begin{array}{ccccccccccccc}

0&<&\frac{3-\sqrt{6}}{3}&<&3-\sqrt{6}&<&\frac{2}{3}&<&\frac{3}{2}&<&
\frac{3+\sqrt{6}}{3}&<&3+\sqrt{6}~.\\ \\ 

&&0.18\ldots&&0.55\ldots&&0.66\ldots&&1.5&&1.81\ldots&&5.44\ldots 
\end{array}$$
For the polynomial $T$, we obtain the following list:

$$\begin{array}{ccccccccc}2a-3&,&3-6a+a^2&,&-1+6a-3a^2&,&-2a+3a^2&,&-a^2\\ \\ 
3/2&&3\pm \sqrt{6}&&(3\pm \sqrt{6})/3&&0~,~2/3&&0~.\end{array}$$
\end{proof}

\begin{proof}[Proof of Proposition~\ref{mainprop4}] 
We observe first that one cannot have $(pos_1, neg_1)=(4,0)$, because then 
the coefficient of $x^3$ in $P'$ (and hence the coefficient of $x^4$ in $P$) 
must be negative. Therefore we have to consider four cases.
\vspace{1mm}

{\em Case 1.} $(pos_1, neg_1)=(0,4)$. Hence $(pos_2, neg_2)=(0,3)$, 
$(pos_3, neg_3)=(0,2)$ and $(pos_4, neg_4)=(0,1)$. There are six possibilities 
for $(pos_0, neg_0)$, and their relizability results as follows: for 
$(0,5)$ and $(1,4)$ (resp. for $(0,3)$ and $(1,2)$ or for $(0,1)$ and $(1,0)$) 
from Proposition~\ref{mainprop1} (resp. from Remarks~\ref{remshyp}  
or Proposition~\ref{mainprop2}).
\vspace{1mm}

{\em Case 2.} $(pos_1, neg_1)=(1,3)$. Hence $pos_0=0$ or $1$, 
see (\ref{Rolle1}). By condition (\ref{Rolle1}), there are two possibilities:
\vspace{1mm}

{\em Case 2a.} $(pos_2, neg_2)=(0,3)$, $(pos_3, neg_3)=(0,2)$ and 
$(pos_4, neg_4)=(0,1)$. There are seven possible values of 
$(pos_0, neg_0)$. For five of them we find out that: 

{\em i)} $(2,3)$ and $(1,4)$ are realizable by Proposition~\ref{mainprop1}; 

{\em ii)} $(1,2)$ is realizable by Remarks~\ref{remshyp}; 

{\em iii)} $(0,1)$ and $(1,0)$ are realizable by Proposition~\ref{mainprop2}. 

To deal with the sixth possibility 
$(pos_0, neg_0)=(0,3)$ we use Lemma~\ref{lm23}. Consider the polynomial 
$S$ with $a\in (0,(3-\sqrt{6})/3)$, and its deformation 
$S_1:=S+\varepsilon (x^2+x)$, where $\varepsilon >0$ is sufficiently small. 
The polynomial $S_1$ has a root at $-1$ at which the first derivative 
is negative. 
Hence to the left and right of this root there are two more negative roots 
(because $S_1(0)=a^2>0$). On the other hand $S_1$ has no positive roots 
(because for $x>0$, one has $S(x)\geq 0$ and $x^2+x>0$). The roots of $S_1$ are 
close to the roots of $S$, so $S_1$ has a complex conjugate pair close to $a$ 
and realizes the sixth possibility. 

The last of the seven possibilities for $(pos_0, neg_0)$ is $(2,1)$. 
We consider again the polynomial 
$S$ with $a\in (0,(3-\sqrt{6})/3)$. Hence $S_2:=S-\varepsilon$ has two real 
positive roots close to $a$ and a simple negative root close to $-1$. For 
$0<\eta \ll \varepsilon$, the polynomial $S_3:=S_2-\eta x$ has two real 
positive roots close to $a$ and a simple negative root close to $-1$; its 
derivative has two simple roots close to $-1$ and a simple root close to $a$. 
The fourth root of $S_2'$ must also be real, and as the constant term 
of $S_2'$ is negative, this root must be negative. Thus the seventh possibility 
is realizable by the polynomial $S_3$.    
 
{\em Case 2b.} $(pos_2, neg_2)=(1,2)$, $(pos_3, neg_3)=(0,2)$ and 
$(pos_4, neg_4)=(0,1)$ or $(pos_2, neg_2)=(1,2)$, $(pos_3, neg_3)=(1,1)$ and 
$(pos_4, neg_4)=(0,1)$ (we consider the two possibilities together). 
The pair $(pos_0, neg_0)$ can take the following values: 

{\em i)} $(1,4)$ or $(2,3)$ -- the cases are realizable 
by Proposition~\ref{mainprop1};

{\em ii)} $(1,2)$ -- the case is realizable by Remarks~\ref{remshyp}; 

{\em iii)} $(0,1)$ or $(1,0)$ -- the cases are realizable 
by Proposition~\ref{mainprop2};  

{\em iv)} $(0,3)$ -- for $(pos_3, neg_3)=(0,2)$, the case is realizable by the 
polynomial

$$\begin{array}{ccl}
G&:=&(x+1.01)(x+1)(x+0.99)((x-0.3)^2+0.01)\\ \\ 
&=&x^5+2.40x^4+1.2999x^3-0.50004x^2-0.299950x+0.099990\end{array}$$

$$\begin{array}{cll}
{\rm roots:}&-1.01 ~,~ -1 ~,~ -0.99 ~;& -1.0\ldots ~,~ -0.9\ldots ~,~ 
-0.2\ldots ~,~ 
0.2\ldots ~;\\  
&-1.0\ldots ~,~ -0.5\ldots ~,~ 0.09\ldots ~;& -0.7\ldots ~,~ -0.1\ldots ~;
\end{array}$$

%
for $(pos_3, neg_3)=(1,1)$, the case is realizable by the 
polynomial

$$\begin{array}{ccl}
H&:=&(x+1.01)(x+1)(x+0.99)((x-0.6)^2+0.01)\\ \\ 
&=&x^5+1.80x^4-0.2301x^3-1.48998x^2-0.089917x+0.369963\end{array}$$

$$\begin{array}{cll}
{\rm roots:}&-1.01 ~,~ -1 ~,~ -0.99 ~;& 
-1.0\ldots ~,~ -0.9\ldots ~,~ -0.03\ldots ~,~ 0.5\ldots ~;\\ 
&-1.0\ldots ~,~ -0.4\ldots ~,~ 
0.3\ldots ~;&-0.7\ldots ~,~ 0.03\ldots ~.\end{array}$$

%

{\em v)} $(2,1)$ -- for $(pos_3, neg_3)=(0,2)$, the case is realizable by the 
polynomial $K:=x^5+20x^4+0.6x^3-5x^2-x+0.5$ 

$$\begin{array}{cll}
{\rm roots:}&-19.9\ldots ~,~ 0.2\ldots ~,~ 0.4\ldots ~;&-15.9\ldots ~,~ 
-0.30\ldots ~,~ 
-0.10\ldots ~,~ 0.38\ldots ~;\\ &-11.9\ldots ~,~ -0.2\ldots ~,~ 0.1\ldots ~;&
-7.9\ldots ~,~ -0.007\ldots ~;\end{array}$$

%
for $(pos_3, neg_3)=(1,1)$, the case is realizable by the polynomial 

$$\begin{array}{ccl}
L&:=&((x+1.01)(x+1)(x+0.99)+0.1)((x-0.6)^2-0.01)\\ \\ 
&=&x^5+1.80x^4-0.2501x^3-1.44998x^2-0.269915x+0.384965\end{array}$$

$$\begin{array}{cll}
{\rm roots:}&-1.4\ldots ~,~ 0.5\ldots ~,~ 0.7\ldots ~;&
-1.1\ldots ~,~ -0.76\ldots ~,~-0.09\ldots ~,~ 0.6\ldots ~;\\ 
&-1.0\ldots ~,~ -0.4\ldots ~,~ 0.3\ldots ~;&-0.7\ldots ~,~ 0.03\ldots ~.
\end{array}$$
%
%

{\em Case 3.} $(pos_1, neg_1)=(2,2)$. There are two cases to consider:

{\em Case 3a.} $(pos_2, neg_2)=(2,1)$, $(pos_3, neg_3)=(1,1)$ and 
$(pos_4, neg_4)=(0,1)$. (One cannot have $(pos_3, neg_3)=(2,0)$, because in this 
case the coefficient of $x$ in $P'''$ hence the one of $x^4$ in $P$ 
must be negative.) There are eight possible values of $(pos_0, neg_0)$:

{\em i)} $(3,2)$ or $(2,3)$ -- realizability follows from 
Proposition~\ref{mainprop1};

{\em ii)} $(0,1)$ or $(1,0)$ -- realizability results from 
Proposition~\ref{mainprop2}; 

{\em iii)} $(3,0)$ or $(1,2)$ -- realizability is deduced 
from Lemma~\ref{lm23} as follows. Consider for some fixed 
$a\in (3/2, (3+\sqrt{6})/3)$ the polynomial $T$ and its deformation 

$$T_{\varepsilon}:=(x-1)(x-1-\varepsilon )(x-1+\varepsilon )(x+a)^2~~~,~~~
0<\varepsilon \ll 1~.$$
It has two critical values attained for some $x\in (1-\varepsilon ,1)$ and 
for some $x\in (1,1+\varepsilon )$. These values are $O(\varepsilon )$. Hence 
one can choose $\varepsilon$ and $\eta >0$ small enough so that the polynomial 
$T_{\varepsilon}+\eta$ (resp. $T_{\varepsilon}-\eta$) realizes the SAP with 
$(pos_0, neg_0)=(3,0)$ (resp. with $(pos_0, neg_0)=(1,2)$). 

{\em iv)} $(2,1)$ -- we realize the SAP by the polynomial 

$$N:=x^5+2x^4-60x^3+0.05x^2+x+5~.$$ 

$$\begin{array}{cll}
{\rm roots:}&-8.8\ldots ~,~ 0.4\ldots ~,~ 6.8\ldots ~;&-6.8\ldots ~,~ 
-0.07\ldots ~,~ 0.07\ldots ~,~ 5.2\ldots ~;\\ 
&-4.8\ldots ~,~ 0.0002\ldots ~,~ 3.6\ldots ~;&-2.8\ldots ~,~ 2.0\ldots ~.
\end{array}$$ 
%

{\em v)} $(0,3)$ -- we realize the SAP by the polynomial 

$$D:=x^5+0.01x^4-1.9990x^3+0.059990x^2+0.99940005x+0.0000019999$$

$$\begin{array}{cll}
{\rm roots:}&-1.1\ldots ~,~ -0.8\ldots ~,~ -0.000002\ldots ~;&
-1.0\ldots ~,~ -0.4\ldots ~,~ 0.4\ldots ~,~ 0.9\ldots ~;\\ 
&-0.7\ldots ~,~ 0.01\ldots ~,~ 0.7\ldots ~;&-0.4\ldots ~,~ 0.4\ldots ~.
\end{array}$$

%

{\em Case 3b.} $(pos_2, neg_2)=(1,2)$, $(pos_3, neg_3)=(0,2)$ and 
$(pos_4, neg_4)=(0,1)$ or $(pos_2, neg_2)=(1,2)$, $(pos_3, neg_3)=(1,1)$ and 
$(pos_4, neg_4)=(0,1)$ (we consider the two possibilities in parallel). 
There are seven possible values for $(pos_0, neg_0)$, the same as in Case 3a. 

{\em i)} For $(3,2)$, $(2,3)$, $(0,1)$ and $(1,0)$, the answers 
why these cases are realizable are the same as in Case 3a. 

{\em ii)} For $(3,0)$ and   
$(1,2)$, we use Lemma~\ref{lm23}. Consider the polynomial $T$ with 
$a>3+\sqrt{6}$ (for $(pos_3, neg_3)=(0,2)$) or $a\in ((3+\sqrt{6})/3$, 
$3+\sqrt{6})$ (for $(pos_3, neg_3)=(1,1)$). The cases are realizable 
by the polynomials $T_{\varepsilon}\pm \eta$ as in Case 3a. 

{\em iii)} For $(2,1)$, and when $(pos_3, neg_3)=(1,1)$, the case is 
realizable by the polynomial 

$$\Lambda :=x^5+0.2x^4-6x^3-0.05x^2+0.01x+0.5~.$$

$$\begin{array}{cll}
{\rm roots:}&-2.5\ldots ~,~ 0.4\ldots ~,~ 2.3\ldots ~;&
-1.9\ldots ~,~ -0.02\ldots ~,~ 0.02\ldots ~,~ 1.8\ldots ~;\\ &
-1.4\ldots ~,~ -0.002\ldots ~,~ 1.2\ldots ~;&-0.81\ldots ~,~ 0.73\ldots ~.
\end{array}$$

%
For $(2,1)$, and when $(pos_3, neg_3)=(0,2)$, 
we realize the case by the polynomial 

$$\Xi :=x^5+2.25x^4+1.0166666666x^3-0.45x^2+0.025x+0.0015~.$$

$$\begin{array}{cll}
{\rm roots:}&-0.03\ldots ~,~ 0.13\ldots ~,~ 0.18\ldots ~;& 
-1.0\ldots ~,~ -0.9\ldots ~,~ 0.03\ldots ~,~ 0.1\ldots ~;\\ &
-0.9\ldots ~,~ -0.4\ldots ~,~ 0.09\ldots ~;&-0.7\ldots ~,~ -0.1\ldots ~.
\end{array}$$


%

{\em iv)} For $(0,3)$, and when $(pos_3, neg_3)=(1,1)$, we realize the case 
by a deformation of the polynomial 
$S$ from Lemma~\ref{lm23} with $a\in (2/3, 3/2)$, namely 

$$S_{\varepsilon}:=(x+1-\varepsilon )(x+1)(x+1+\varepsilon )
((x-a)^2+\varepsilon )~~~,~~~
0<\varepsilon \ll 1~.$$
For $(0,3)$,  and when $(pos_3, neg_3)=(0,2)$, we realize the case by the 
polynomial 

$$\Phi :=x^5+2.4x^4+0.481x^3-0.8510x^2+0.08529x+0.01729~.$$ 

$$\begin{array}{cll}
{\rm roots:}&-1.9 ~,~ -1 ~,~ -0.1~;&-1.6\ldots ~,~ -0.6\ldots ~,~ 
0.05\ldots ~,~ 
0.2\ldots ~;\\ &-1.2\ldots ~,~ -0.3\ldots ~,~ 0.1\ldots ~;&
-0.9\ldots ~,~ -0.05\ldots ~.\end{array}$$


{\em Case 4.}  $(pos_1, neg_1)=(3,1)$. Hence the SP is of the form 
$(+,+,-,+,-,\pm )$, because the SP defined by $P'$ must have three sign 
changes. Thus $(pos_2, neg_2)=(2,1)$, 
$(pos_3, neg_3)=(1,1)$ and $(pos_4, neg_4)=(0,1)$. There are seven 
possibilities for $(pos_0, neg_0)$ out of which 
$(4,1)$ and $(3,2)$ (resp. $(2,1)$) are realizable by 
Proposition~\ref{mainprop1} (resp. by Remarks~\ref{remshyp}) while 
the realizability of $(0,1)$ and $(1,0)$ results from 
Proposition~\ref{mainprop2}. We realize the case $(pos_0, neg_0)=(1,2)$ by 
the polynomial 
$$U:=x^5+x^4-9.01x^3+10.97x^2-4.05x-0.01~.$$

$$\begin{array}{cll}
{\rm roots:}&-4.0\ldots ~,~ -0.002\ldots ~,~ 1.2\ldots ~;&
-3.0\ldots ~,~ 0.2\ldots ~,~ 0.8\ldots ~,~ 1.0\ldots ~;\\ 
&-2.1\ldots ~,~ 0.5\ldots ~,~ 1.0\ldots ~;&-1.1\ldots ~,~ 0.7\ldots~.
\end{array}$$
%
The case $(pos_0, neg_0)=(3,0)$ is not realizable, 
see (\ref{exd5}) in Example~\ref{exnotrealiz}.

\end{proof}

\begin{proof}[Proof of Proposition~\ref{mainprop5}]
We are considering neither the cases with 
$(pos_0, neg_0)=(0,1)$ or $(1,0)$ (which 
have been treated by Proposition~\ref{mainprop2}) nor the ones with 
$pos_0+neg_0=5$ (see Proposition~\ref{mainprop1}) nor the ones with 
$(pos_2, neg_2)=(0,1)$ or $(1,0)$ 
(which have been settled by Proposition~\ref{mainprop3}). Therefore we are 
going to limit ourselves to the situations in which $pos_0+neg_0=3$ 
and $pos_2+neg_2=3$. It is impossible to have $(pos_2, neg_2)=(3,0)$, because 
this would mean that the coefficient of $x^2$ in $P''$ (hence the one of $x^4$ 
in $P$) must be negative. So 
three cases have to be examined (defined by $(pos_2, neg_2)$):
\vspace{1mm}

{\em Case A.}  $(pos_2, neg_2)=(0,3)$. Hence $(pos_3, neg_3)=(0,2)$ and 
$(pos_4, neg_4)=(0,1)$. Observe first that one cannot have 
$(pos_1, neg_1)=(2,0)$, because then $P''$ should have at least one positive 
root. Therefore $(pos_1, neg_1)=(0,2)$ or $(1,1)$. 
For $(pos_1, neg_1)=(0,2)$, we realize the cases $(pos_0, neg_0)=(0,3)$ and 
$(pos_0, neg_0)=(1,2)$ by the 
polynomials $\tilde{P}$ and $P_*$ respectively: 

$$
\tilde{P}~:=~x^5+20x^4+40x^3+5x^2+x+0.5$$

$$\begin{array}{cll}
{\rm roots:}&-17.7\ldots ~,~ -2.1\ldots ~,~ -0.2\ldots ~;&-14.3\ldots ~,~ 
-1.5\ldots ~;\\ &-10.9\ldots ~,~ -1.0\ldots ~,~ -0.04\ldots ~;&
-7.4\ldots ~,~ -0.5\ldots ~;\end{array}$$

$$P_*~:=~x^5+20x^4+40x^3+5x^2+x-0.5~.$$
For $k\geq 1$, the roots of $P_*^{(k)}$ and $\tilde{P}^{(k)}$ are the same due to 
$\tilde{P}-P_*\equiv 1$. The roots of $P_*$ equal $-17.7\ldots$, 
$-2.1\ldots$ and~$0.1\ldots$.


For $(pos_1, neg_1)=(1,1)$, we realize the cases $(pos_0, neg_0)=(2,1)$ 
and $(pos_0, neg_0)=(1,2)$ 
by the polynomials $\tilde{Q}$ and $Q_*$: 

$$\tilde{Q}~:=~x^5+100x^4+20x^3+0.5x^2-x+0.005$$

$$\begin{array}{cll}
{\rm roots:}&-99.7\ldots ~,~ 0.005\ldots ~,~ 0.1\ldots ~;&-79.8\ldots ~,~ 
0.09\ldots ~;\\ &-59.8\ldots ~,~ -0.09\ldots ~,~ -0.009\ldots ~;&
-39.9\ldots ~,~ -0.05\ldots ~;\end{array}$$

$$Q_*~:=~x^5+30x^4+20x^3+5x^2-x-0.5$$

$$\begin{array}{cll}
{\rm roots:}&-29.3\ldots ~,~ -0.3\ldots ~,~ 0.2\ldots ~;&-23.4\ldots ~,~ 
0.06\ldots ~;\\ &-17.6\ldots ~,~ -0.18\ldots ~,~ -0.15\ldots ~;&
-11.8\ldots ~,~ -0.1\ldots ~.\end{array}$$
\vspace{1mm}


{\em Case B.}  $(pos_2, neg_2)=(1,2)$, $(pos_3, neg_3)=(1,1)$ and 
$(pos_4, neg_4)=(0,1)$. 

{\em Case B1.} $(pos_1, neg_1)=(0,2)$. We realize the case 
$(pos_0, neg_0)=(0,3)$ by the polynomial 

$$J_{\sharp}~:=~x^5+9x^4-0.8x^3-0.0073x^2+96x+36$$

$$\begin{array}{cll}
{\rm roots:}&-8.9\ldots ~,~-2.2\ldots ~,~-0.3\ldots ~;&-7.2\ldots ~,~
-1.4\ldots ~;\\ &-5.4\ldots ~,~-0.002\ldots ~,~0.04\ldots ~;&-3.6\ldots ~,~
0.02\ldots ~.\end{array}$$
We realize the case 
$(pos_0, neg_0)=(1,2)$ by the polynomial 
$$V_{\flat}~:=~x^5+9x^4-0.8x^3-0.0073x^2+96x-36$$

$$\begin{array}{cll}
{\rm roots:}&-8.9\ldots ~,~-2.5\ldots ~,~0.3\ldots ~;&-7.2\ldots ~,~
-1.4\ldots ~;\\ &-5.4\ldots ~,~-0.002\ldots ~,~0.04\ldots ~;&-3.6\ldots ~,~
0.02\ldots ~.\end{array}$$

{\em Case B2.} $(pos_1, neg_1)=(1,1)$.
We realize the case $(pos_0, neg_0)=(2,1)$ by the polynomial 
$$P_{\sharp}~:=~x^5+0.2x^4-6x^3-0.05x^2-0.1x+0.05$$ 
$$\begin{array}{cll}
{\rm roots:}&-2.5\ldots ~,~0.1\ldots ~,~2.3\ldots ~;&-1.9\ldots ~,~
1.8\ldots ~;\\ &-1.4\ldots ~,~-0.002\ldots ~,~1.2\ldots ~;&
-0.8\ldots ~,~0.7\ldots ~.\end{array}$$
We realize the case $(pos_0, neg_0)=(1,2)$ by the polynomial 

$$P_{\circ}~:=~x^5+0.2x^4-6x^3-0.05x^2-0.1x-0.05$$ 
$$\begin{array}{cll}
{\rm roots:}&-2.5\ldots ~,~-0.1\ldots ~,~2.3\ldots ~;&-1.9\ldots ~,~1.8\ldots 
~;\\ &-1.4\ldots ~,~-0.002\ldots ~,~1.2\ldots ~;&-0.8\ldots ~,~0.7\ldots ~.
\end{array}$$

{\em Case B3.} $(pos_1, neg_1)=(2,0)$.
To realize the case $(pos_0, neg_0)=(3,0)$ we consider the polynomial 
$$W_{\flat}~:=~x^5+4.4x^4-19.295x^2+13.22x-1.1295$$ 
$$\begin{array}{cll}
{\rm roots:}&0.1~,~0.6\ldots ~,~1.3\ldots ~;&0.3\ldots ~,~1.0\ldots ~;\\ &
-2.2\ldots ~,~-1.1\ldots ~,~0.7\ldots ~;&-1.7\ldots ~,~0~.\end{array}$$
As we see, all real roots of $W_{\flat}^{(k)}$, $k\leq 4$, are simple. Hence for 
$\varepsilon >0$ sufficiently close to $0$, the polynomial 
$W_{\flat}-\varepsilon x^3$ realizes this case.  

To realize the case $(pos_0, neg_0)=(2,1)$ we construct first the polynomial 
$$W_{\sharp}~:=~x^5+4.6x^4-17.495x^2+8.74x+1.0485$$

$$\begin{array}{cll}
{\rm roots:}&-0.1~,~0.6\ldots ~,~1.3\ldots ~;&0.2\ldots ~,~1.0\ldots ~;\\ &
-2.4\ldots ~,~-0.9\ldots ~,~0.7\ldots ~;&-1.84~,~0~.\end{array}$$
We realize the case by the polynomial $W_{\sharp}-\varepsilon x^3$. 
\vspace{1mm}

{\em Case C.} $(pos_2, neg_2)=(1,2)$, $(pos_3, neg_3)=(0,2)$ and 
$(pos_4, neg_4)=(0,1)$. 

{\em Case C1.} $(pos _1, neg_1)=(0,2)$. 
We realize the case $(pos_0, neg_0)=(0,3)$ by the polynomial 

$$P_{\flat}~:=~x^5+9x^4+3x^3-0.73x^2+96x+36$$
$$\begin{array}{cll}
{\rm roots:}&-8.4\ldots ~,~-2.5\ldots ~,~-0.3\ldots ~;&
-6.8\ldots ~,~-1.6\ldots ~;\\ &-5.2\ldots ~,~-0.2\ldots ~,~0.05\ldots ~;&
-3.5\ldots ~,~-0.08\ldots ~.\end{array}$$
We realize the case $(pos_0, neg_0)=(1,2)$ by the polynomial 
$$T_{\flat}~:=~x^5+20x^4+80x^3-0.02x^2+x-0.5$$

$$\begin{array}{cll}
{\rm roots:}&-14.4\ldots ~,~-5.5\ldots ~,~0.1\ldots ~;&-11.9\ldots ~,~
-4.0\ldots ~;\\ &-9.4\ldots ~,~-2.5\ldots ~,~0.00008\ldots ~;&-6.8\ldots ~,~-1.1\ldots ~.
\end{array}$$

{\em Case C2.} $(pos _1, neg_1)=(1,1)$. 
We realize the case $(pos_0, neg_0)=(2,1)$ by the polynomial 
$$S_{\flat}~:=~x^5+9x^4+3x^3-0.73x^2-96x+36$$
$$\begin{array}{cll}
{\rm roots:}&-8.7\ldots ~,~0.3\ldots ~,~1.8\ldots ~;&-6.9\ldots ~,~1.2\ldots ~;
\\ &-5.2\ldots ~,~-0.2\ldots ~,~0.05\ldots ~;&-3.5\ldots ~,~-0.08\ldots ~.
\end{array}$$
We realize the case $(pos_0, neg_0)=(1,2)$ by the polynomial 

$$U_{\flat}~:=~x^5+20x^4+0.06x^3-0.05x^2-x-0.5$$
$$\begin{array}{cll}
{\rm roots:}&-19.9\ldots ~,~-0.3\ldots ~,~0.4\ldots ~;&-15.9\ldots ~,~
0.2\ldots ~;\\ &-11.9\ldots ~,~-0.02\ldots ~,~0.01\ldots ~;&-7.9\ldots ~,~
-0.0007\ldots ~.\end{array}$$

{\em Case C3.} $(pos _1, neg_1)=(2,0)$. We realize the case 
$(pos _0, neg_0)=(3,0)$ by the polynomial $W_{\flat}+\varepsilon x^3$, 
and the case $(pos _0, neg_0)=(2,1)$ by the polynomial 
$W_{\sharp}+\varepsilon x^3$, with $W_{\flat}$ and $W_{\sharp}$ as defined in Case~B3.
One cannot have $(pos _0, neg_0)=(1,2)$ or $(0,3)$, see~(\ref{Rolle1}).
\vspace{1mm}

{\em Case D.} $(pos_2, neg_2)=(2,1)$. One cannot have $(pos_3, neg_3)=(2,0)$, 
because then the coefficient of $x$ in $P'''$ (hence the one of $x^4$ in $P$) 
should be negative. Therefore $(pos_3, neg_3)=(1,1)$ and 
$(pos_4, neg_4)=(0,1)$. The possibility $(pos_1, neg_1)=(2,0)$ 
has not to be considered -- it gives rise to the four SAPs (\ref{exd45}). 
So we have to treat two possibilities; 

{\em Case D1.} $(pos_1, neg_1)=(1,1)$. Hence $(pos_0, neg_0)=(1,2)$ or $(2,1)$, 
see (\ref{Rolle1}). We realize the case 
$(pos_0, neg_0)=(1,2)$ by the polynomial 

$$P_{\dagger}~:=~x^5+0.2x^4-6x^3+0.05x^2-0.01x-0.5$$ 

$$\begin{array}{cll}
{\rm roots:}&-2.5\ldots ~,~ -0.4\ldots ~,~ 2.3\ldots ~;&-1.9\ldots ~,~ 
1.8\ldots ~;\\ 
&-1.4\ldots ~,~ 0.002\ldots ~,~1.2\ldots ~;&-0.8\ldots ~,~ 0.7\ldots ~.
\end{array}$$
We realize the case $(pos_0, neg_0)=(2,1)$ by the polynomial 

$$K_{\flat}~:=~x^5+9x^4-0.8x^3+0.0073x^2-96x+36$$

$$\begin{array}{cll}
{\rm roots:}&-9.2\ldots ~,~ 0.3\ldots ~,~ 1.9\ldots ~;&-7.3\ldots ~,~ 
1.3\ldots ~;\\ 
&-5.4\ldots ~,~ 0.003\ldots ~,~ 0.04\ldots ~;&-3.6\ldots ~,~ 
0.02\ldots ~.\end{array}$$


{\em Case D2.} $(pos_1, neg_1)=(0,2)$. Hence $(pos_0, neg_0)=(0,3)$ or $(1,2)$, 
see (\ref{Rolle1}). We realize the case 
$(pos_0, neg_0)=(0,3)$ by the polynomial 
$$J_{\flat}~:=~x^5+9x^4-0.8x^3+0.0073x^2+96x+36$$

$$\begin{array}{cll}
{\rm roots:}&-8.9\ldots ~,~-2.2\ldots ~,~-0.3\ldots ~;&
-7.2\ldots ~,~-1.4\ldots ~;\\ &-5.4\ldots ~,~0.003\ldots ~,~0.04\ldots ~;&
-3.6\ldots ~,~0.02\ldots ~.\end{array}$$
We realize the case $(pos_0, neg_0)=(1,2)$ by the polynomial 
$$K_{\sharp}~:=~x^5+9x^4-0.8x^3+0.0073x^2+96x-36$$

$$\begin{array}{cll}
{\rm roots:}&-8.9\ldots ~,~-2.5\ldots ~,~0.3\ldots ~;&-7.2\ldots ~,~
-1.4\ldots ~;\\ &-5.4\ldots ~,~0.003\ldots ~,~0.04\ldots ~;&-3.6\ldots ~,~
0.02\ldots ~.\end{array}$$

\end{proof}

\begin{proof}[Proof of part (3) of Theorem~\ref{maintm}]
There are three possible values for the sum $pos_1+neg_1$, namely, 
$0$, $2$ and $4$. If $pos_1+neg_1=0$, then $pos_0+neg_0=1$ (see (\ref{Rolle1})) 
and the realizability of such a case results from Proposition~\ref{mainprop3}. 
If $pos_1+neg_1=2$ or $4$, then realizability follows from 
Proposition~\ref{mainprop4} or~\ref{mainprop5}.
\end{proof}

\begin{proof}[Proof of Proposition~\ref{A(d)}]
For $d=2$ and $3$ the proposition is to be checked straightforwardly. Suppose 
that $d\geq 4$. Denote by $h_{d,m}$ the number of SAPs with 
$(pos_0, neg_0)=(0,m)$. Set 
$h_{d,m}:=0$ for $m>d$. Hence $h_{d,d}=1$ and 
$$h_{d,m}=\left\{ \begin{array}{ll}h_{d,m+2}&{\rm if~}d{\rm ~is~even~and~} m=0\\ 
h_{d,m+2}+h_{d-1,m-1}&{\rm in~all~other~cases}\end{array}\right. ~.$$
This can be deduced from conditions (\ref{Rolle1}) and (\ref{Rolle2}). 
Thus if $d$ is even, 
then one deduces from the above formulas that  

$$h_{d,2}=h_{d,0}=h_{d-1,d-1}+h_{d-1,d-3}+\cdots +h_{d-1,1}=A(d-1)~,$$
and as $h_{d,d}=1>0$, one obtains $A(d)>2A(d-1)$. If $d$ is odd, then 

$$\begin{array}{lc}h_{d,3}=h_{d-1,d-1}+h_{d-1,d-3}+\cdots +h_{d-1,2}&
{\rm and}\\ h_{d,1}=h_{d-1,d-1}+h_{d-1,d-3}+\cdots +h_{d-1,2}+h_{d-1,0}=A(d-1)&.
\end{array}$$
As $d-1$ is even, one has $h_{d-1,2}=h_{d-1,0}$, so $h_{d,3}>A(d-1)/2$ and 
$A(d)>h_{d,3}+h_{d,1}>3A(d-1)/2$.
\end{proof}

\end{document}